\documentclass{tac}

\usepackage{amssymb,amscd, amsmath, latexsym,mathrsfs}
\usepackage[all]{xy}
\usepackage{tikz}

\newcommand{\sarl}[2]{\ar@<1.6pt>[#2]^-{#1}\ar@<-1.6pt>[#2]}
\newcommand{\sarr}[2]{\ar@<1.6pt>[#2]\ar@<-1.6pt>[#2]_-{#1}}
\newcommand{\sarlb}[2]{\ar@/^4pt/@<1.6pt>[#2]^-{#1}\ar@/^4pt/@<-1.6pt>[#2]}
\newcommand{\sarrb}[2]{\ar@/_4pt/@<1.6pt>[#2]\ar@/_4pt/@<-1.6pt>[#2]_-{#1}}
\newcommand{\sarc}[2]{\ar@<0pt>@{}[#2]|-{#1}\ar@<1.6pt>[#2]\ar@<-1.6pt>[#2]}
\newcommand{\sarlh}[3]{\ar@/#3/@<1.6pt>[#2]^-{#1}\ar@/#3/@<-1.6pt>[#2]}
\newcommand{\sarrh}[3]{\ar@/#3/@<1.6pt>[#2]\ar@/#3/@<-1.6pt>[#2]_-{#1}}
\newcommand{\sarch}[3]{\ar@/#3/@<0pt>@{}[#2]|-{#1}\ar@/#3/@<1.6pt>[#2]\ar@<-1.6pt>[#2]}
\newdir{ >}{
  @{}*!/-10pt/@{>} }
\newcommand{\Cc}{\ensuremath{\mathbb{C}}}
\newcommand{\Nc}{\ensuremath{\mathcal{N}}}
\newcommand{\ot}{\otimes}
\newcommand{\pullbackcorner}[1][dr]{\save*!/#1+1.2pc/#1:(1,-1)@^{|-}\restore}
\newcommand{\pushoutcorner}[1][dr]{\save*!/#1-1.2pc/#1:(-1,1)@^{|-}\restore}
\newcommand{\eqvee}{\scalebox{0.75}{\rotatebox{90}{$\eqslantless$}}}

\newtheorem{convention}{Convention}
\newtheorem{theorem}{Theorem}[section]
\newtheorem{lemma}[theorem]{Lemma}
\newtheorem{proposition}[theorem]{Proposition}
\newtheorem{corollary}[theorem]{Corollary}

\newtheorem{remark}[theorem]{Remark}

\title{The Zassenhaus lemma in star-regular categories}
\author{Olivette Ngaha Ngaha and Florence Sterck}
\copyrightyear{2019}

\address{Institut de Recherche en Math\'ematique et Physique, Universit\'e catholique de Louvain, Chemin du Cyclotron 2, 1348 Louvain-la-Neuve, Belgium}
\eaddress{omngaha@yahoo.fr\\ florence.sterck@uclouvain.be}

\thanks{The first author was supported by a \emph{Bourse de coop\'eration au d\'eveloppement} de
l'Universit\'e catholique de Louvain and the second one by a FRIA doctoral grant of the \textit{Communaut\'e fran\c{c}aise de Belgique}.}

\begin{document}
\maketitle

\begin{abstract}
The Noether Isomorphism Theorems and the Zassenhaus Lemma from group theory have a 
non-pointed version in a suitable categorical context first considered by W. Tholen in his PhD thesis. This article leads to a unification of these results with the ones in the pointed categorical context previously considered by O.~Wyler, by working in the framework of \emph{star-regular} categories introduced by M.~Gran, Z.~Janelidze and A.~Ursini. Some concrete examples of categories where these results hold are examined in detail.
\end{abstract}

\section*{Introduction}

A finitely complete category $\Cc$ is said to be \emph{normal} \cite{ZJane} when it is pointed, any arrow $f \colon X\rightarrow Y$ has a factorisation $f=me$ with $e$ a normal epimorphism and $m$ a monomorphism, and these factorisations are pullback-stable. Equivalently, a normal category is a pointed regular category in which every regular epimorphism is a normal epimorphism. 

In \cite{oW} O.~Wyler investigated a large class of pointed categories in which the (Noether) isomorphism theorems and the Zassenhaus Lemma, well known in group theory, remain valid. These categories are very close to the \emph{ideal determined categories} introduced by G.~Janelidze, L.~M\'arki, W.~Tholen and A.~Ursini in \cite{JaneMarThoUrs}, which are normal categories (with binary coproducts) in which the normal image of a normal monomorphism is again a normal monomorphism. 

In his doctoral dissertation \cite{wT}, W.~Tholen gave a non-pointed version of the isomorphism theorems and of the Zassenhaus Lemma in a \emph{regular} category \cite{BGO} satisfying some suitable conditions. This interesting result was mentioned in the talk \cite{Ttalk} in 2011.

A natural question was then to look for a common extension of these two results, by considering a context which has already shown to be useful for unifying and comparing results from pointed and non-pointed categorical algebra, introduced by M.~Gran, Z.~Janelidze and A.~Ursini \cite{GranJaneUrs}. This context is essentially based on the notion of an ideal $\Nc$ of morphisms in the sense of C. Ehresmann \cite{Ehr}, which also plays a role in the work of M.~Grandis in Homological Algebra \cite{Grandis}. The \emph{pointed context} is recovered by choosing $\Nc$ to be the class of zero morphisms of a pointed category, while the non-pointed context, called the \emph{total context} is obtained by choosing $\Nc$ to be the class of all morphisms of a category. 

In the present paper, we unify the results recalled above by proving the validity of the isomorphism theorems and the Zassenhaus Lemma in any \emph{star-regular} category \cite{GranJaneUrs} satisfying some suitable properties. In \cite{GranJaneUrs} it was shown  that the notion of star-regular category is a context that unifies a regular category in the total context, and a normal category in the pointed context.

The paper is organised in four sections. In the first section, we briefly recall the main notions and results from the theory of ``star-relations'' needed in the following. In the second section, we prove the isomorphism theorems (Proposition \ref{DIT} and \ref{DQIT}) in a star-regular category, by introducing a condition which is important for our work, called \emph{property $(*)$}, which holds true in both, the total and in the pointed contexts. In the third section, the Zassenhaus Lemma is established in a suitable categorical context (Theorem \ref{ZL}), that will be more general than both the contexts considered in \cite{wT} and in \cite{oW}. Finally, in the last section, we describe the Zassenhaus Lemma in the category of cocommutative Hopf algebras. This is a new result (Theorem \ref{ZL in coco}) in the theory of (cocommutative) Hopf algebras.

\section*{Acknowledgements}

The authors would like to warmly thank their supervisor Professor Marino Gran for all the invaluable advice in the realization of this paper. Many thanks also to Tomas Everaert for many useful suggestions and comments of the subject on this paper. The authors also thank the referee for his/her useful remarks.



\section{Star-regular categories}\label{starregularcat}
This section is devoted to recalling some basic aspects concerning ``star-relations'' in a category with finite limits that will be useful later on. The reader will find further details in \cite{GranJaneUrs}.

A class $\Nc$ of morphisms in $\Cc$ is an \emph{ideal} of morphisms when for two composable morphisms $f\colon A\rightarrow B$ and $g\colon B\rightarrow C$ of $\Cc$, if $f\in \Nc$ or $g\in \Nc$, then the composite $gf$ belongs to $\mathcal{N}$. A category $\Cc$ equipped with such an ideal $\mathcal{N}$ of morphisms is called a \emph{multi-pointed category} \cite{GranJaneUrs}. This notion of a category equipped with an ideal of morphisms was introduced by C.~Ehresmann \cite{Ehr} and was used by M.~Grandis in \cite{Grandis}.  

A \emph{star} in a multi-pointed category $(\Cc,\Nc)$ is a pair of parallel morphisms \[\tau=(\tau_1,\tau_2)\colon T\rightrightarrows X\] such that $\tau_1\in \mathcal{N}$; it is a \emph{monic star} when, moreover, the star $\tau$ has the property that for any two morphisms $u,v\colon U\rightarrow T$ of $\Cc$ such that $\tau u=\tau v$ (this means that $\tau_1 u=\tau_1 v$ and $\tau_2 u=\tau_2 v$), then we have $u=v$. In other words, the pair $(\tau_1,\tau_2)$ is jointly monomorphic. The $\mathcal{N}$-\emph{kernel} of a morphism $f\colon X\rightarrow Y$ is a morphism $k\colon \mathcal{N}ker(f)\rightarrow X$ of $\Cc$ such that the composite $fk\in \mathcal{N}$ and which is universal with this property:  for any other morphism $g\colon L\rightarrow X$ such that $fg\in \mathcal{N}$, there
exists a unique morphism $\mu\colon L\rightarrow \mathcal{N}ker(f)$ such that $k\mu=g$:
 $$\xymatrix{\mathcal{N}ker(f)\ar[r]^{\;\; \; \; \; k}&X\ar[r]^f&Y\\L\ar[ru]_g\ar[u]^{\mu} }$$
An $\mathcal{N}$-\emph{kernel} of a morphism is always a monomorphism. 

When $\mathcal{N}$-kernels exist, any relation $\rho=(\rho_1,\rho_2)\colon R\rightrightarrows X$ on an object $X$ in $\Cc$ gives rise to a monic star $\rho^*=(\rho_1k,\rho_2k)\colon R^* \rightrightarrows X$ where $k\colon R^* := \mathcal{N}ker(\rho_1)\rightarrow R$ is the $\mathcal{N}$-kernel of $\rho_1$. $\rho^*$ is the largest subrelation (as monic star) of $\rho$. In particular, if $\Delta_X\colon X\rightrightarrows X$ is the discrete equivalence relation on $X$, we have $\Delta^*_X=(k_X,k_X)\colon X^* \rightrightarrows X$, where $k_X\colon X^*\rightarrow X$, is the $\mathcal{N}$-kernel of $1_X$.

A \emph{kernel star} of a morphism
$f\colon X\rightarrow Y$ is a star $\kappa=(\kappa_1,\kappa_2)\colon  F^* \rightrightarrows X$ such that $f\kappa_1=f\kappa_2$ and, universal with this property, i.e. for any other star
$\tau =(\tau_1,\tau_2) \colon T\rightrightarrows X$ such that $f\tau_1=f\tau_2$,
there exists a unique morphism $\mu\colon T\rightarrow K$ such that
$\kappa \mu=\tau$:
\[
\xymatrix{F^*\ar@<2pt>[r]^{\kappa}\ar@<-2pt>[r]_{}&X\ar[r]^f
&Y\\T\ar@<2pt>[ru]^{}\ar@<-2pt>[ru]_{\tau}
\ar[u]^{\mu}&}
\]
A kernel star is always a monic star. In the presence of $\mathcal N$-kernels, it is easy to see that the kernel star of an arrow $f \colon X \rightarrow Y$ is given by $Eq(f)^*\rightrightarrows X$ where $Eq(f)\rightrightarrows X$ is the kernel pair of $f$. \begin{convention} $Eq(f)^*\rightrightarrows X$ and $Eq(f)\rightrightarrows X$ will be denoted $F^*\rightrightarrows X$ and $F \rightrightarrows X$ respectively.\end{convention}

A \emph{star-pullback} is a diagram
\[
\xymatrix{S\ar@<2pt>[r]^{\sigma}\ar@<-2pt>[r]_{}\ar[d]_g&
X\ar[d]^{f}\\T\ar@<2pt>[r]^{}\ar@<-2pt>[r]_{\tau}&Y}
\]
of stars and morphisms such that
$f\sigma_1=\tau_1g$ and $f\sigma_2=\tau_2g$, and for any other star $\sigma':
S'\rightrightarrows X$ and any other morphism $g'\colon S'\rightarrow T$
such that $f\sigma'_1=\tau_1 g'$ and $f\sigma'_2=\tau_2 g'$, there exists a unique morphism
$h\colon S'\rightarrow S$ such that $\sigma_1 h=\sigma'_1, \sigma_2 h=\sigma'_2$ and $gh=g'$:

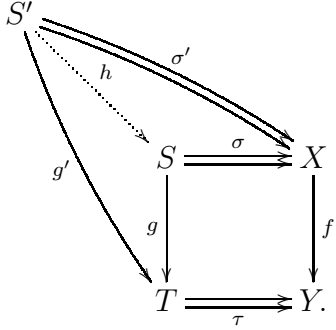
\[
\xymatrix@=40pt{S'\sarlh{\sigma'}{drr}{^5pt}\ar@/_5pt/[ddr]_-{g'}\ar@{.>}[dr]^-{h} & & \\ &  S \ar[d]_-{g}\sarl{\sigma}{r} & X\ar[d]^-{f} \\ & T\sarr{\tau}{r} & Y.}
\]

These notions of ``star-relations" give the well-known notions as follows:
\begin{enumerate}
\item Any pointed category $\Cc$ with zero object $0$ can be seen as a multi-pointed category by choosing for the ideal $\mathcal N$ the class of zero morphisms: in this case, one refers to it as the \emph{pointed context}. In this context, a star is essentially an arrow of the category (=the second component, since the first one is a zero morphism), while a monic star is a monomorphism. The kernel star of a morphism is the classical kernel of this morphism and the notion of star-pullback gives back the well-known notion of a pullback.
\item If $\Nc$ is the class of all morphisms in a category $\Cc$, one calls it the \emph{total context}: in this case, a star is a pair of parallel morphisms and a monic star is a jointly monomorphic pair of parallel morphisms. The kernel star of a morphism gives the notion of kernel pair of this morphism, and a star-pullback becomes the notion of joint-pullback (see \cite{Bou02}, for instance).
\end{enumerate}

The following lemma gives us a way to get an important tool in our work: namely the factorisation of a star as a regular epimorphism followed by a monic star (uniquely up to isomorphism). 

\begin{lemma}\label{nkernel}\cite{GranJaneUrs} Let $(\Cc,\Nc)$ be a multi-pointed category with $\Nc$-kernels. Then, for a composable pair of morphisms $\xymatrix@=16pt{A\ar[r]^f&B\ar[r]^g&C}$ such that the composite $gf\in \Nc$ and $f$ is a regular epimorphism, then $g$ also belongs to $\Nc$.
\end{lemma}
\begin{proof}
The morphism $f$ factors through the $\Nc$-kernel $k$ of $g$. Since $f$ is a regular epimorphism, $k$ is actually an iso, so that $g = (gk)k^{-1} \in \Nc$ follows.
\end{proof}
In the pointed context, this lemma says that if $gf=0$ and $f$ is a regular epimorphism, then $g=0$. In the total context, this lemma becomes trivial.

\begin{convention}
$\mathbb C$ is a finitely complete regular category.\end{convention} 

Thus, when $(\Cc,\Nc)$ is a regular multi-pointed category with $\Nc$-kernels, any star  $(\lambda_1,\lambda_2)\colon X\rightrightarrows Y$ factorises (up to isomorphism) as a regular epimorphism followed by a monic star
\[
\xymatrix@=32pt{X\ar@{->>}[rd]_e\ar@<2pt>[rr]^{\lambda_1}\ar@<-2pt>[rr]_{\lambda_2}&&Y\\&Z
\ar@<2pt>[ru]^{\mu_1}\ar@<-2pt>[ru]_{\mu_2}}
\]
Indeed, if we consider the usual factorisation  in $\Cc$
\[
\xymatrix@=22pt{X\ar@{->>}[rd]_e\ar[rr]^{(\lambda_1,\lambda_2)}&&Y\times Y\\&Z
\ar[ru]_{(\mu_1,\mu_2)}}
\] of $(\lambda_1,\lambda_2)$ by a regular epimorphism followed by a monomorphism, we use Lemma \ref{nkernel} to conclude that $(\mu_1,\mu_2)\colon Z\rightrightarrows Y$ is a monic star.

For a morphism $f\colon X\rightarrow Y$ and a star $\lambda\colon U\rightrightarrows X$, we call $f(\lambda)$ \emph{the image of $\lambda$ along $f$} which is the (monic) star $f(\lambda)\colon f(U)\rightrightarrows Y$ obtained in the factorisation of the star $f\lambda$ as a regular epimorphism followed by a monic star
\[
\xymatrix{U\ar[r]^{f'}\ar@<2pt>[d]^{}\ar@<-2pt>[d]_{\lambda}&f(U)
\ar@<2pt>[d]^{f(\lambda)}
\ar@<-2pt>[d]_{}\\X\ar[r]_f&Y}
\] The \emph{star-inverse image} of a star $\sigma\colon T\rightrightarrows Y$ along a morphism $f\colon X\rightarrow Y$ denoted by $f^{-1}(\sigma)^*\colon f^{-1}(T)^*\rightrightarrows X$ is the star obtained in the following star-pullback
\[
\xymatrix{f^{-1}(T)^*\ar[r]^{}\ar@<2pt>[d]^{}\ar@<-2pt>[d]_{f^{-1}(\sigma)^*}&T\ar@<2pt>[d]^{\sigma}
\ar@<-2pt>[d]_{}\\X\ar[r]_f&Y}
\] Equivalently, $f^{-1}(\sigma)^*$ is obtained by taking the star of the usual inverse image of the star $\sigma$ along $f$ i.e., $f^{-1}(\sigma)^*=(f^{-1}(\sigma))^*$ or similary, $f^{-1}(T)^*=(f^{-1}(T))^*$.

\begin{remark} If $f$ is a monomorphism and $\sigma\colon T\rightrightarrows Y$ is a 
monic star, by abuse of notation, we shall write $f^{-1}(\sigma)^*=(\sigma\cap f)^*$ or with sources, $f^{-1}(T)^*=(T\cap (X\times X))^*.$
\end{remark}


\begin{definition}\label{starregularcat}\cite{GranJaneUrs}
A \emph{star}-\emph{regular} category is a regular multi-pointed
category $(\Cc,\mathcal{N})$ with $\mathcal{N}$-kernels in which
every regular epimorphism is a coequaliser of a star.
\end{definition}

In the total context, a star-regular category is simply a regular
category. In the pointed context, a star-regular category is the same as
a normal category in the sense of \cite{ZJane}, i.e. a pointed regular
category in which any regular epimorphism is a normal epimorphism.

In a star-regular category, any kernel star $\kappa\colon F^*\rightrightarrows A$ has a coequaliser, written as $f\colon A\rightarrow A/F^*$.

The category $\mathsf{Grp}$ of groups is an example of a normal category. In fact, any semi-abelian category \cite{JMT} is normal. Recall that a semi-abelian category is a pointed, Barr-exact and protomodular category with binary coproducts. Thus the categories $\mathsf{Ab}$ of abelian groups, $\mathsf{Mod_R}$ of modules over a ring $R$, $\mathsf{Lie_K}$ of Lie $K$-algebras over a field $K$ are all examples of normal categories. The category $\mathsf{Rng}$ of unitary rings is also an example of star-regular category but this category is not pointed.
\section{Isomorphism theorems}\label{Isomorphisms theorem}

In this section, $(\Cc,\Nc)$ will always be a regular multi-pointed category and sometimes, we shall omit to mention $\Nc$ if there is no confusion.
\begin{definition}\label{diamond}\cite{GJRU}
A \emph{diamond} in a regular multi-pointed category $\Cc$ is a commutative diagram of the form
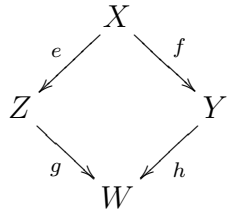
\[
\xymatrix@=20pt{&X\ar[rd]^f\ar[dl]_e\\Z\ar[dr]_g&&Y\ar[dl]^h\\&W}
\]
Such a diamond is :
\begin{enumerate}
\item a \emph{regular diamond} if all arrows are regular epimorphisms;
\item \emph{right saturated} if $f(E^*)=H^*$;
\item \emph{left saturated} if $e(F^*)=G^*$;
\item \emph{saturated} if it is both left saturated and right saturated.
\end{enumerate}
\end{definition}

\begin{definition}\label{saturating}\cite{GJRU}
A morphism $f\colon X \rightarrow Y$ in $\Cc$ is said to be \emph{saturating} if a diamond of the form 
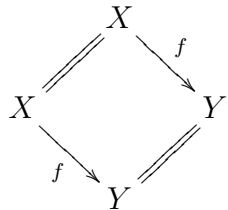
\[
\xymatrix@=20pt{&X\ar[rd]^f\ar@{=}[dl]\\X\ar[dr]_f&&Y\ar@{=}[dl]\\&Y}
\] is right saturated.
\end{definition}
In the total context, saturating morphisms are regular epimorphisms, while in the pointed context all morphisms are saturating.
The following lemma gives a connection between some regular diamonds and saturating regular epimorphisms.
\begin{lemma}
In a regular multi-pointed category $\Cc$, the following are equivalent:
\begin{enumerate}
\item[(i)] any regular epimorphism is saturating;
\item[(ii)] any regular diamond of the form 
\begin{equation}\label{leftsaturated}
\xymatrix@=20pt{&X\ar@{->>}[rd]^f\ar@{->>}[dl]_e\\Z\ar@{->>}[dr]_g&&Y\ar@{=}[dl]\\&Y}
\end{equation}

is right saturated. 
\end{enumerate}
\end{lemma}

\begin{proof}
$(i)\Rightarrow (ii)$. Consider the following diagram
\[
\xymatrix@=36pt{\Delta^*_X\ar@<2pt>[d]\ar@<-2pt>[d]\ar[r]^{\theta}&E^*\ar@<2pt>[d]
\ar@<-2pt>[d]\ar[r]^{f'}&\Delta^*_Y\ar@<2pt>[d]\ar@<-2pt>[d]\\X\ar@{=}[r]\ar@{=}[d]&X
\ar@{->>}[r]^f\ar@{->>}[d]_e&Y\ar@{=}[d]\\X\ar@{->>}[r]_e&Z\ar@{->>}[r]_g&Y}
\] where the bottom right square can be seen as a regular diamond of the form $(1)$. By assumption, since any regular epimorphism is saturating, the composite $f'\theta$ is a regular epimorphism. This implies that $f'$ is a regular epimorphism, as desired.

$(ii)\Rightarrow (i)$. Conversely, given a regular epimorphism $f\colon X\twoheadrightarrow Y$, the right saturation property obviously implies that $f$ is saturating.
\end{proof}

Under the assumption that any regular diamond of the form $(1)$ is left saturated, the following lemma gives a characterization of star-regular categories

\begin{lemma}\label{caractstarregular}
In a regular multi-pointed category $(\Cc, \Nc)$ with $\Nc$-kernels, with coequalisers of kernels stars and such that any regular diamond of the form $(1)$ is left saturated, the following are equivalent:
\begin{enumerate}
\item[(i)] $(\Cc, \Nc)$ is a star-regular category;
\item[(ii)] an arrow $m\colon  U\rightarrow V$ is a monomorphism if and only if $M^*=\Delta_U^*$.
\end{enumerate}
\end{lemma}
\begin{proof}
$(i)\Rightarrow (ii)$. We recall the proof of this implication given in \cite{GN}. Let us then assume $(i)$. It is obvious that if an arrow $m\colon  U\rightarrow V$ is a monomorphism, then $M^*=\Delta_U^*$. For the converse, first note that since $\Cc$ is a star-regular category, the correspondence $\Theta\colon $ KernelPairs$\longrightarrow $ KernelStars mapping any kernel pair $F$ in $\Cc$ to its corresponding star $F^*$ is a bijection. Then, by the injectivity of $\Theta$, $M^*=\Delta_U^*$ implies that $M=\Delta_U$ and thus, $m$ is a monomorphism.

$(ii)\Rightarrow (i)$. It suffices to prove that any regular epimorphism is a coequaliser of a star. Let $f\colon A\twoheadrightarrow B$ be a regular epimorphism and consider the following commutative diagram
\[
\xymatrix@=36pt{F^*\ar[d]_{\tau}\ar@<2pt>[r]^{}\ar@<-2pt>[r]^{}&A\ar@{->>}[r]^f\ar@{->>}[d]_q&B
\ar@{=}[d]\\\Delta^*_{A/_{F^*}}\ar@<2pt>[r]^{}\ar@<-2pt>[r]^{}\ar[d]_{\alpha}&A/_{F^*}\ar[r]_m
\ar@{=}[d]&B\ar@{=}[d]\\M^*\ar@<2pt>[r]^{}\ar@<-2pt>[r]^{}&A/_{F^*}\ar[r]_m&B}
\] where $q$ is a coequaliser of the kernel star $F^*\rightrightarrows A$ of $f$, $m$ is the induced arrow such that $mq=f$. The arrow $m$ is then a regular epimorphism and it remains to prove that $m$ is a monomorphism. Since we assume that any regular diamond of the form $(1)$ is left saturated in $\Cc$, the composite $\alpha \tau$ is a regular epimorphism. This implies that $\alpha$ is a regular epimorphism as well, thus an isomorphism. It follows that $M^*=\Delta^*_{A/_{F^*}}$ and, by assumption, $m$ is a monomorphism.
\end{proof}

\begin{definition}\label{property}
A regular multi-pointed category $(\Cc,\Nc)$ with $\Nc$-kernels and with coequalisers of kernel stars is said to satisfy \emph{the property $(*)$} if the following condition holds in $\Cc$:
for any kernel star $\xymatrix@=16pt{F^*\ar@<2pt>[r]\ar@<-2pt>[r]&A}$ with coequaliser $f\colon A\rightarrow A/F^*$, and for any other kernel star $\xymatrix@=16pt{G^*\ar@<2pt>[r]\ar@<-2pt>[r]&A}$ such that $F^*\subseteq G^*$, then $\xymatrix@=16pt{f(G^*)\ar@<2pt>[r]\ar@<-2pt>[r]&A/F^*}$ is a kernel star.
\[
\xymatrix@=32pt{&G^*\ar@<2pt>[d]\ar@<-2pt>[d]\ar@{->>}[r]&f(G^*)
\ar@<2pt>[d]\ar@<-2pt>[d]\\F^*\ar[ru]\ar@<2pt>[r]\ar@<-2pt>[r]&A\ar@{->>}[r]_(0.4)f&A/F^*}
\]

\end{definition}

We have the connection between some regular diamonds and the property $(*)$ in the following
\begin{proposition}\label{equivalence}
The following conditions are equivalent in a star-regular category $\Cc$.
\begin{enumerate}
\item[(i)]$\Cc$ satisfies the property $(*)$;
\item[(ii)]any regular diamond of the form (\ref{leftsaturated}) is left saturated.
\end{enumerate}
\end{proposition}

\begin{proof}$(i)\Rightarrow (ii).$
Consider the following diagram
\[
\xymatrix@=38pt{&F^*\ar@<2pt>[d]^{f_1}\ar@<-2pt>[d]_{f_2}\ar@/_0.8pc/[dl]_
{m}\\G^*\ar@<2pt>[r]^{g_1}\ar@<-2pt>[r]_{g_2}\ar@{->>}[d]_{f'}&A\ar@{->>}[r]^g\ar@{->>}[d]_f&C\ar@{=}[d]\\f(G^*)\ar@<2pt>[r]^(0.6){r_1}\ar@<-2pt>[r]_(0.6){r_2}&B\ar@{->>}[r]_q&C}
\] where the right-hand square is a regular diamond of the form (\ref{leftsaturated}). We are going to prove that it is left saturated, i.e. $f(G^*)\cong Q^*.$ For this, we consider the kernel star $F^*\rightrightarrows A$ of $f$. The fact that $qf=g$ gives us the induced monomorphism $m.$ Then, by assumption $(i)$, $\xymatrix@=19pt{f(G^*)\ar@<2pt>[r]^(0.7){r_1}\ar@<-2pt>[r]_(0.7){r_2}&B}$ is a kernel star. Moreover, since $f'$ is a regular epimorphism and $g$ is the coequaliser of $(g_1,g_2)$, $q$ is the coequaliser of $(r_1,r_2)$. The fact that in a star-regular category any kernel star is the kernel star of its coequaliser allows us to conclude that $f(G^*)\cong Q^*.$

 $(ii)\Rightarrow (i).$ Conversely, let us first consider the triangle of the diagram here above where $F^*$, $G^*$ are kernel stars and $m$ is a monomorphism. The arrows $f$ and $g$ are coequalisers of $F^*$ and $G^*$, respectively, and $(r_1,r_2)f'$ is the factorisation (regular epimorphism, monic-star) of the star $(fg_1,fg_2)$. Since $gf_1=gf_2$, we have the induced arrow $q$ such that $qf=g$ which is then a regular epimorphism since $g$ is a regular epimorphism. We so construct the right-hand square of the diagram above which is a regular diamond of the form (\ref{leftsaturated}). Accordingly, by $(ii)$, this regular diamond is left saturated and this implies that $f(G^*)\cong Q^*$ is the kernel star of $q$.
\end{proof}
In both the pointed and the total contexts, any regular diamond of the form (\ref{leftsaturated}) is left saturated. Indeed, by Lemma $2.5.$ in \cite{GranJaneRod}, since the left hand side square of this diagram 
\[
\xymatrix@=38pt{G^*\ar@<2pt>[r]^{g_1}\ar@<-2pt>[r]_{g_2}\ar[d]_{f'}&A\ar@{->>}[r]^g\ar@{->>}[d]_f&C\ar@{=}[d]\\Q^*\ar@<2pt>[r]^(0.6){r_1}\ar@<-2pt>[r]_(0.6){r_2}&B\ar@{->>}[r]_q&C}
\] is a star-pullback, this easily follows from the fact that regular epimorphisms are pullback-stable in the pointed context, whereas they are joint-pullback stable in the total context. This implies that the property $(*)$ is always satisfied in both the pointed and the total contexts.


\begin{definition}\label{supremum1}
Let $\Cc$ be a star-regular category. For any kernel star $\kappa\colon \xymatrix@=8pt{F^*\ar@<2pt>[r]\ar@<-2pt>[r]&A}$ with coequaliser $f\colon A\rightarrow A/F^*$, and any monomorphism $m\colon M\rightarrow A$, we define their \emph{asymmetric join} 
\begin{center}
$F^*\,\eqvee_{\!A}\,  M=f^{-1}(f(M))$ 
\end{center}
as the subobject $\kappa\,\eqvee_{\!A}\, m\colon F^*\,\eqvee_{\!A}\, M\rightarrow A$ of $A$ in the following pullback
\[\xymatrix@=32pt{M\ar@/_0.8pc/[ddr]_m\ar@/^1pc/@{->>}[rrd]\ar@{-->}[dr]\\&F^*
\,\eqvee_{\!A}\, M \pushoutcorner \ar[d]^{\displaystyle \kappa\,\eqvee_{\!A}\, m}\ar@{->>}[r]&f(M)\ar[d]^{f(m)}\\F^*\ar@<2pt>[r]\ar@<-2pt>[r]&A\ar@{->>}[r]_f&A/F^*}
\]
\end{definition}
This concept appears in the total context in \cite{wT}, while in the pointed context it appears in \cite{oW}. Note that in the pointed context, when $\Cc$ is a normal category, then for any kernel $k\colon K\rightarrow A$ with cokernel $f\colon A\rightarrow A/K$, and any monomorphism $m\colon M\rightarrow A$, $K\,\eqvee_{\!A}\,M=f^{-1}(f(M))$ is a subobject of $A$ containing both $K$ and $M$. 

 Let us point out that the notions of asymmetric join in the pointed and total contexts coincide if the base category is normal, since we can express the quotient $A \rightarrow A/F$ as $A \rightarrow A/K$ where $K$ is the kernel of $f$. 

The following lemma gives another characterisation of the asymmetric join in any \emph{semi-abelian} category \cite{JMT}, i.e. a pointed, Barr-exact, protomodular category with binary coproducts.
\begin{lemma}\label{lemma sup}
If $\Cc$ is a semi-abelian category, $K\,\eqvee_{\!A}\, M=f^{-1}(f(M))$ is exactly the supremum of $K$ and $M$ (as subobjects of $A$).
\end{lemma}
\begin{proof}
Let $L$ be another subobject of $A$ containing both $K$ and $M$ and consider the following diagram
\[
\xymatrix@=38pt{K\,\eqvee_{\!A}\, M\ar@{.>}[dr]\ar@/^1.5pc/[rr]^{}&M\ar[r]\ar[d]\ar[l]&f(M)\ar@{-->}[d]\\K\ar[r]^{n}\ar@{=}[d]\ar[u]&L\ar[d]^l\ar[r]^{p}&f(L)\ar[d]\\K\ar[r]_k&A\ar[r]_f&A/K}
\] where the vertical dotted arrow is induced by the fact that regular epimorphisms are orthogonal under monomorphisms. One can verify that the monomorphism $n$ is in fact the kernel of the normal epimorphism $p$, and the protomodularity of $\Cc$ allows us to conclude that the bottom right-hand square is a pullback (see for instance section 2.3 ($PM_0$) in \cite{JMT}). This implies the existence of the diagonal dotted arrow.
\end{proof}
In the particular case of groups, $f \eqvee_{\!A} M = f^{-1}(f(M))$ is exactly the binary product $KM = MK$.

In the following two propositions, we use the terminology of Mac Lane and Birkhoff in their book \emph{Algebra} \cite{MB}.

We may now state the diamond isomorphism theorem:
\begin{proposition}\label{DIT}[Diamond isomorphism theorem]
Let $\Cc$ be a star-regular category, $\kappa\colon \xymatrix@=12pt{F^*\ar@<2pt>[r]\ar@<-2pt>[r]&A}$ a kernel star and $m\colon M\rightarrow A$ a monomorphism. There is an isomorphism
\begin{center}
$\dfrac{M}{(F^*\cap (M\times M))^*}\cong \dfrac{F^*\,\eqvee_{\!A}\, M}{(F^*\cap ((F^*\,\eqvee_{\!A}\, M)\times (F^*\,\eqvee_{\!A}\, M)))^*}$
\end{center}
of subobjects of $A/F^*$. The diamond isomorphism theorem may be summarized in the diagram below
\[
\xymatrix{&F^*\,\eqvee_{\!A}\, M\ar@{-}[rd]\ar@{-}[ld]\\(F^*\cap ((F^*\,\eqvee_{\!A}\, M)\times (F^*\,\eqvee_{\!A}\, M)))^*\ar@{-}[rd]&&M\ar@{-}[dl]\\&(F^*\cap (M\times M))^*}
\]

\end{proposition}

\begin{proof} In order to prove this isomorphism, we are going to 
show that both $$\xymatrix@=16pt{(\kappa\cap (\kappa\,\eqvee_{\!A}\, m))^*\colon (F^*\cap ((F^*\,\eqvee_{\!A}\, M)\times (F^*\,\eqvee_{\!A}\, M)))^*\ar@<2pt>[r]\ar@<-2pt>[r]&F^*\,\eqvee_{\!A}\, M}$$ and  $$\xymatrix@=16pt{(\kappa\cap m)^*\colon (F^*\cap (M\times M))^*\ar@<2pt>[r]\ar@<-2pt>[r]&M}$$ are kernel stars with the same quotient. For this, let us set $L=(F^*\cap (M\times M))^*$, $T=(F^*\cap ((F^*\,\eqvee_{\!A}\, M)\times (F^*\,\eqvee_{\!A}\, M)))^*$ and then consider the following diagram 
\[
\xymatrix@=32pt{L\ar@<2pt>[rr]^{\displaystyle (\kappa\cap m)^*}
\ar@<-2pt>[rr]_{}\ar[dr]\ar@/_1.5pc/[ddr]&&M
\ar@/^1pc/@{->>}[rrd]^{f'}\ar[rd]\ar@/_1pc/[rdd]_(0.6){m}\\&T \ar@<2pt>[rr]^{\displaystyle (\kappa\cap (\kappa\,\eqvee_{\!A}\, m))^*}\ar@<-2pt>[rr]_{}\ar[d]_{l}&&F^*\,\eqvee_{\!A}\, M \ar@{->>}[r]^(0.6){f''}\ar[d]^{m''}&f(M)\ar[d]^{m'}\\&F^*
\ar@<2pt>[rr]^{\kappa}\ar@<-2pt>[rr]_{}&&A\ar@{->>}[r]_{f}&A/_{F^*}}
\] where $f$ is the coequaliser of its kernel star $F^*.$ It is easy to check that $$\xymatrix@=16pt{(\kappa\cap (\kappa\,\eqvee_{\!A}\, m))^*\colon T\ar@<2pt>[r]\ar@<-2pt>[r]&F^*\,\eqvee_{\!A}\, M}$$ is the kernel star of $fm''=m'f''$ and since $m'$ is a monomorphism, we conclude that it is the kernel star of the regular epimorphism $f''$. The star-regularity of $\Cc$ allows us to assert that $f''$ is the coequaliser of $\xymatrix@=18pt{T\ar@<2pt>[r]\ar@<-2pt>[r]&F^*\,\eqvee_{\!A}\, M}$. Accordingly, $f(M)\cong \dfrac{F^*\,\eqvee_{\!A}\, M}{T}$ as subobjects of $A/F^*.$ In the same way, $\xymatrix@=16pt{(\kappa\cap m)^*\colon L\ar@<2pt>[r]\ar@<-2pt>[r]&M}$ is the kernel star of the regular epimorphism $f'$, then $f'$ is the coequaliser of $\xymatrix@=16pt{(\kappa\cap m)^*\colon L\ar@<2pt>[r]\ar@<-2pt>[r]&M}$. Thus $f(M)\cong {M}/{L}$ as subobjects of $A/F^*$.
\end{proof}

In the pointed context, in a normal category $\Cc$, if $k\colon K\rightarrow A$ is a kernel, $m\colon M\rightarrow A$ a monomorphism, then the diamond isomorphism theorem asserts that 
\begin{center}
$ \dfrac{M}{K\cap M} \cong \dfrac{K\,\eqvee_{\!A}\, M}{K\cap (K\,\eqvee_{\!A}\, M)}$
\end{center}
Moreover, since $K\subset K\,\eqvee_{\!A}\, M$, then $K\cap (K\,\eqvee_{\!A}\, M)=K,$ and we find the well-known diamond isomorphism theorem 
\[\dfrac{M}{K\cap M} \cong \dfrac{K\,\eqvee_{\!A}\, M}{K}
\] 

A non-pointed version of the diamond isomorphism theorem was given in \cite{wT}.

 \begin{remark}If $\Cc$ is a normal category, then in the pointed and total contexts the diamond isomorphism theorem gives two equivalent statements. We already know that there is essentially only one form of the asymmetric join in a normal category. If we look at the expression $\dfrac{M}{(F^*\cap (M\times M))^*}$, in the total context this means that we are looking at $\dfrac{M}{(F\cap (M\times M))}$, and this is equivalent to the quotient $\dfrac{M}{K \cap M}$, where $K\rightarrow A$ is the kernel of $f$. 
Moreover, the second member $\dfrac{F^*\,\eqvee_{\!A}\, M}{(F^*\cap ((F^*\,\eqvee_{\!A}\, M)\times (F^*\,\eqvee_{\!A}\, M)))^*}$ of the isomorphism is isomorphic to $\dfrac{K\,\eqvee_{\!A}\, M}{K}$. Then the isomorphisms in Proposition \ref{DIT} simplifies to \[\dfrac{M}{K\cap M} \cong \dfrac{K\,\eqvee_{\!A}\, M}{K}.
\]  in the total context. \end{remark}

\begin{proposition}\label{DQIT}[Double quotient isomorphism theorem]\label{thirdiso}
Let $(\Cc,\Nc)$ be a star-regular category satisfying the property $(*)$. Then for two kernel stars $\xymatrix@=16pt{F^*\ar@<2pt>[r]\ar@<-2pt>[r]&A}$,$\xymatrix@=16pt{G^*\ar@<2pt>[r]\ar@<-2pt>[r]&A}$ such that $F^*\subseteq G^*$, the isomorphism 
$$A/G^*\cong \dfrac{A/F^*}{f(G^*)}$$ holds in $\Cc$.
\end{proposition}

\begin{proof}Consider the following diagram
\[
\xymatrix@=38pt{&F^*\ar@<2pt>[d]^{f_1}\ar@<-2pt>[d]_{f_2}\ar@/_0.8pc/[dl]_
{m}\\G^*\ar@<2pt>[r]^{g_1}\ar@<-2pt>[r]_{g_2}\ar[d]_{f'}&A\ar@{->>}[r]^g\ar@{->>}[d]_f&A/G^*\ar@{=}[d]\\f(G^*)\ar@<2pt>[r]^(0.5){r_1}\ar@<-2pt>[r]_(0.5){r_2}&A/F^*\ar@{->>}[r]_q&A/G^*}
\] where $f$ and $g$ are coequalisers of their kernel stars $\xymatrix@=16pt{F^*\ar@<2pt>[r]\ar@<-2pt>[r]&A}$ and $\xymatrix@=16pt{G^*\ar@<2pt>[r]\ar@<-2pt>[r]&A}$ respectively, $q$ the induced arrow such that $qf=g$ and $(r_1,r_2)f'$ the factorisation (regular epimorphism, monic-star) of the star $(fg_1,fg_2)$. Assuming that the property $(*)$ is satisfied, it follows that $\xymatrix@=19pt{f(G^*)\ar@<2pt>[r]^(0.5){r_1}\ar@<-2pt>[r]_(0.5){r_2}&A/F^*}$ is a kernel star. Moreover, since $g$ is a regular epimorphism, then $q$ is a regular epimorphism as well and one checks that $q$ is in fact the coequaliser of $(fg_1,fg_2)=(r_1f',r_2f')$. Hence $q$ is the coequaliser of $(r_1,r_2)$ since $f'$ is a regular epimorphism and $g$ is the coequaliser of $(g_1,g_2)$. Finally, we have the desired isomorphism $A/G^*\cong \dfrac{A/F^*}{f(G^*)}$.
\end{proof}
In the pointed context, 
if $\Cc$ is a normal category, then for two normal monomorphisms $k\colon K\rightarrow A$, $l\colon L\rightarrow A$ such that $K\subseteq L$,  we have:
\[
A/L \cong \dfrac{A/K}{f(L)}
\]
where $f$ is a cokernel of the kernel $K\rightarrow A.$ Moreover, it is easy to verify that $f(L)\cong L/K$. Then we find the well-known double quotient isomorphism theorem in a normal category \cite{GE}. A non-pointed version of the double quotient isomorphism theorem  appears in \cite{wT}. In particular, in $\mathsf{Grp}$ the category of groups, $\mathsf{Vect}_K$ the category of $K$-vector spaces on a field $K$, etc. Proposition \ref{DIT} and Proposition \ref{DQIT} give, in the pointed and total contexts, the well-known second and third isomorphism theorems. In the category of unitary rings, the pointed context doesn't make sense (the category is not pointed), but the total context gives us the classical statement of isomorphism theorems of rings.
 
\section{The Zassenhaus Lemma}
\begin{definition}\label{supremum2}Let $\Cc$ be a star-regular category. 
\begin{enumerate}
\item[1.]We define an asymmetric join of a kernel star $\kappa\colon \xymatrix@=16pt{F^*\ar@<2pt>[r]\ar@<-2pt>[r]&A}$ with coequaliser $f\colon A\rightarrow A/F^*$, and a star  $\tau\colon \xymatrix@=16pt{R\ar@<2pt>[r]\ar@<-2pt>[r]&A}$ by setting
\begin{center}
$F^*\,\eqvee_{\!A}\, R=f^{-1}(f(R))^*$ 
\end{center}
which is the star $\kappa\,\eqvee_{\!A}\, \tau\colon F^*\,\eqvee_{\!A}\, R\rightrightarrows A$, obtained in the following in a star-pullback
\[
\xymatrix@=32pt{R\sarlh{}{ddr}{_5pt}_{\tau}\ar@{-->}[dr]\ar@/^0.8pc/@{->>}[drr]\\&F^*\,\eqvee_{\!A}\, R \pushoutcorner \ar@<2pt>[d]^{\displaystyle \kappa\,\eqvee_{\!A}\, \tau}\ar@<-2pt>[d]\ar@{->>}[r]&f(R)\ar@<2pt>[d]^{f(\tau)}
\ar@<-2pt>[d]\\&A\ar@{->>}[r]_f&A/F^*}
\]

\item[2.] Moreover, if $\Cc$ has pushouts of regular epimorphisms, then for two kernel stars $\xymatrix@=16pt{F^*\ar@<2pt>[r]\ar@<-2pt>[r]&A}$ and $\xymatrix@=16pt{G^*\ar@<2pt>[r]\ar@<-2pt>[r]&A}$ on a same object $A$, we define their \emph{supremum} as kernel stars
\begin{center}
$F^*\vee_A G^*=G^*\vee_A F^*$ 
\end{center}
to be the smallest kernel star on $A$ containing both the kernel stars $F^*$ and $G^*$. $F^*\vee_A G^*$ is the kernel star of the arrow $q=g'f=f'g$ defined by the following pushout
\[
\xymatrix@=32pt{F^*\vee G^*\ar@<2pt>[rd]\ar@<-2pt>[rd]&G^*\ar@<2pt>[d]\ar@<-2pt>[d]\ar@{.>}[l]\\F^*\ar@<2pt>[r]\ar@<-2pt>[r]\ar@{.>}[u]&A\ar@{->>}[r]^f\ar@{->>}[d]_g\ar@{->>}[rd]^q&A/F^*\ar@{->>}[d]^{g'}\\&A/G^*\ar@{->>}[r]_{f'}&D \pullbackcorner}
\]
\end{enumerate}
\end{definition}

The following lemma is crucial to prove the Zassenhaus Lemma in a star-regular category. For the  pointed version of this lemma, see Theorem 5.1 in \cite{oW}.

\begin{lemma}\label{kernel}
Let $\Cc$ be a star-regular category satisfying the
property $(*)$. Then for two kernel stars $\kappa\colon \xymatrix@=16pt{F^*\ar@<2pt>[r]\ar@<-2pt>[r]&A}$, $\sigma\colon \xymatrix@=16pt{G^*\ar@<2pt>[r]\ar@<-2pt>[r]&U}$ and a monomorphism $u\colon U\rightarrow A$ such that  $(F^*\cap (U\times U))^*\subseteq G^*$, we have that
$\xymatrix@=12pt{F^*\,\eqvee_{\!A}\, G^*\ar@<2pt>[r]\ar@<-2pt>[r]&F^*\,\eqvee_{\!A}\, U}$ is a kernel star and 
\[
\dfrac{F^*\,\eqvee_{\!A}\, U}{F^*\,\eqvee_{\!A}\, G^*} \cong \dfrac{U}{G^*} 
\]
\end{lemma}
\begin{proof}
Consider the following diagram

\[
\xymatrix@=24pt{&G^*\ar[rd]\ar@{->>}[rr]^{}\ar@<2pt>[dd]^{\sigma}
\ar@<-2pt>[dd]&&f'(G^*)\ar@<2pt>[dd]\ar@<-2pt>[dd]\ar@{=}[rd]\\&&F^*\,\eqvee_{\!A}\, G^*\ar@{.>}[dd]<2pt>\ar@{.>}[dd]<-2pt>\ar@{->>}[rr]&&f'(G^*)\ar@<2pt>[dd]\ar@<-2pt>[dd]\\(F^*\cap (U\times U))^*\ar[ruu]
\ar[dd]\ar@<2pt>[r]^(0.7){(\kappa\cap u)^*}\ar@<-2pt>[r]&U\ar[dr]
\ar[dd]_u\ar@{->>}[rr]^(0.7){f'}&&f(U)\ar[dd]\ar@{=}[dr]\\&&F^*\,\eqvee_{\!A}\, U \ar[dd]^(0.7){\displaystyle \kappa\,\eqvee_{\!A}\, u}\ar@{->>}[rr]&&f(U)\ar[dd]\\F^*\ar@<2pt>[r]^{\kappa}\ar@<-2pt>[r]&A\ar@{->>}[rr]^(0.7)f\ar@{=}[dr]&&A/F^*\ar@{=}[dr]\\&&A\ar@{->>}[rr]_f&&A/F^*}
\] where the bottom square and the bottom front square of the outer cube are pullbacks. By the uniqueness of the factorisation of a star as a regular epimorphism followed by a monic-star, we have the isomorphism $f(G^*)\cong f'(G^*)$ as stars on $A/F^*$. This implies that the front diagram of the outer cube is a star-pullback by Definition \ref{supremum2}.1. Then, the upper front square is a star-pullback as well. Since $\xymatrix@=16pt{(F^*\cap (U\times U))^*\ar@<2pt>[r]\ar@<-2pt>[r]&U}$ is a kernel star of the regular epimorphism $f'$ then, by the property $(*)$, $\xymatrix@=16pt{f'(G^*)\ar@<2pt>[r]\ar@<-2pt>[r]&f(U)}$ is again a kernel star. Thus $\xymatrix@=16pt{F^*\,\eqvee_{\!A}\, G^*\ar@<2pt>[r]\ar@<-2pt>[r]&F^*\,\eqvee_{\!A}\, U}$ is the kernel star of the diagonal $F^*\,\eqvee_{\!A}\, U\longrightarrow f(U)\longrightarrow \dfrac{f(U)}{f(G^*)}$ which is a regular epimorphism as a  composite of two regular epimorphisms. Moreover, the double quotient isomorphism theorem gives us the isomorphism ${f(U)}/{f'(G^*)} \cong {U}/{G^*}$. By the star-regularity of $\Cc$, we finally have $$\dfrac{F^*\,\eqvee_{\!A}\, U}{F^*\,\eqvee_{\!A}\, G^*} \cong \dfrac{f(U)}{f(G^*)}\cong \dfrac{U}{G^*}.$$
\end{proof}

\begin{theorem}[Zassenhaus's Lemma]\label{ZL}
Let $\Cc$ be a star-regular category with pushouts of regular epimorphisms that satisfies the property $(*)$. Then, for two kernel stars $\kappa\colon \xymatrix@=16pt{F^*\ar@<2pt>[r]\ar@<-2pt>[r]&U,}$ $\sigma\colon \xymatrix@=16pt{G^*\ar@<2pt>[r]\ar@<-2pt>[r]&V}$ and two monomorphisms $u\colon U\rightarrow A$, $v\colon V\rightarrow A,$ we have the isomorphisms 
\begin{eqnarray*}
&&\dfrac{F^*\,\eqvee_{U} (U\cap V)}{F^*\,\eqvee_{U}((F^*\cap (V\times V))^*\vee_{U\cap V} (G^*\cap (U\times U))^*)}\\
&\cong & \dfrac{U\cap V}{(F^*\cap (V\times V))^*\vee_{U\cap V} (G^*\cap (U\times U))^*} \\
                       & \cong &\dfrac{G^*\,\eqvee_{V} (U\cap V)}{G^*\,\eqvee_{V} ((F^*\cap (V\times V))^*\vee_{U\cap V} (G^*\cap (U\times U))^*)} 
\end{eqnarray*}
\end{theorem}

\begin{proof}
First of all, let us construct $M=(F^*\cap (V\times V))^*\,\vee_{U\cap V} (G^*\cap (U\times U))^*$. For this, consider the diagram
\[
\xymatrix@=32pt{M\ar@<2pt>[rd]\ar@<-2pt>[rd]&(G^*\cap (U\times U))^*\ar@<2pt>[d]^{(\sigma\cap u')^*}\ar@<-2pt>[d]\ar@{-->}[l]\ar[r]&G^*\ar@<2pt>[d]\ar@<-2pt>[d]\\(F^*\cap (V\times V))^*\ar@<2pt>[r]^(0.5){(\kappa\cap v')^*}\ar@<-2pt>[r]\ar[d]
\ar@{-->}[u]&U\cap V\ar[r]^{u'}\ar[d]_{v'}&V\ar[d]^v\\F^*\ar@<2pt>[r]^{\kappa}\ar@<-2pt>[r]&U\ar[r]_u&A}
\]
where the bottom right square is a pullback by definition of intersection of monomorphisms and both the bottom left and top right squares are star-pullbacks. This implies that both $\xymatrix@=16pt{(F^*\cap (V\times V))^*\ar@<2pt>[r]\ar@<-2pt>[r]&U\cap V}$ and $\xymatrix@=16pt{(G^*\cap (U\times U))^*\ar@<2pt>[r]\ar@<-2pt>[r]&U\cap V}$ are kernel stars. By Definition \ref{supremum2}.2, we have the existence of $\xymatrix@=16pt{M\ar@<2pt>[r]\ar@<-2pt>[r]&U\cap V}$ which is a kernel star of the regular epimorphism $q$ obtained in the following pushout
\begin{eqnarray}\label{pushout}
\xymatrix@=32pt{U\cap V\ar@{->>}[r]^{f_1}\ar@{->>}[d]_{g_1}\ar@{->>}[rd]^q&f(U\cap V)\ar@{->>}[d]^{\overline{g_1}}\\g(U\cap V)\ar@{->>}[r]_{\overline{f_1}}&T}
\end{eqnarray} where the regular epimorphisms $f_1$ and $g_1$ are the coequalisers of $\xymatrix@=10pt{(F^*\cap (V\times V))^*\ar@<2pt>[r]\ar@<-2pt>[r]&U\cap V}$ and of $\xymatrix@=16pt{(G^*\cap (U\times U))^*\ar@<2pt>[r]\ar@<-2pt>[r]&U\cap V}$, respectively. Thus, we have $T\cong \dfrac{U\cap V}{M}$.

 On one hand, we have the kernel stars $\xymatrix@=16pt{F^*\ar@<2pt>[r]\ar@<-2pt>[r]&U}$, $\xymatrix@=16pt{M\ar@<2pt>[r]\ar@<-2pt>[r]&U\cap V}$ and the monomorphism $U\cap V\rightarrow U$ such that $(F^*\cap (V\times V))^*\subseteq M$. By applying Lemma \ref{kernel}, one obtains the kernel star $\xymatrix@=16pt{F^*\,\eqvee_{U} M\ar@<2pt>[r]\ar@<-2pt>[r]&F^*\,\eqvee_{U} (U\cap V)}$ and the isomorphism
\[
\dfrac{U\cap V}{M}\cong T\cong \dfrac{F^*\,\eqvee_{U} (U\cap V)}{F^*\,\eqvee_{U} M}
\] 
On the other hand, in the same way, we have the kernel stars $\xymatrix@=16pt{G^*\ar@<2pt>[r]\ar@<-2pt>[r]&V}$, $\xymatrix@=16pt{M\ar@<2pt>[r]\ar@<-2pt>[r]&U\cap V}$ and the monomorphisms $U\cap V\rightarrow V$, $(G^*\cap (U\times U))^*\rightarrow M$. According to Lemma \ref{kernel}, $\xymatrix@=16pt{G^*\,\eqvee_{V} M\ar@<2pt>[r]\ar@<-2pt>[r]&G^*\,\eqvee_{V} (U\cap V)}$ is a kernel star and $$\dfrac{U\cap V}{M}\cong T\cong \dfrac{G^*\,\eqvee_{V} (U\cap V)}{G^*\,\eqvee_{V} M}$$
\end{proof}

If in addition $\Cc$ has pushouts of regular epimorphisms along monomorphisms as needed in \cite{wT} (in this case $\Cc$ has pushouts of regular epimorphisms along any morphism), then let us consider the induced monomorphism $y'\colon U\cap V\rightarrowtail  F^*\,\eqvee_{U} (U\cap V)$ and the diagram
\[
\xymatrix@=28pt{U\cap V\,\ar@{>->}[r]^(0.4){y'}\ar@{->>}[d]_{g_1}&F^*\,\eqvee_{U} (U\cap V)\ar@{->>}[d]^{g_1'}\ar@{->>}[r]^(0.6){f'}&f(U\cap V)\ar@{->>}[d]^{\overline{g_1}}\\g(U\cap V)\ar[r]_{y''}&Z\ar@{..>>}[r]_{f''}&T}
\]where the left-hand rectangle is the pushout of $y'$ and $g_1$ and the outer rectangle is the pushout (\ref{pushout}). This implies that the right-hand rectangle is a pushout as well. Take the regular image along $y'$ of $\xymatrix@=16pt{(G^*\cap (U\times U))^*\ar@<2pt>[r]\ar@<-2pt>[r]&U\cap V}$ the kernel star of the regular epimorphism $g_1$, thus $g'_1$ is a coequaliser of $\xymatrix@=12pt{(G^*\cap (U\times U))^*\ar@<2pt>[r]\ar@<-2pt>[r]&F^*\,\eqvee_{U} (U\cap V)}$. 

Furthermore, let us set $L=(F^*\cap (F^*\,\eqvee_{U} (U\cap V)\times F^*\,\eqvee_{U} (U\cap V)))^*$, then $\xymatrix@=16pt{L\ar@<2pt>[r]\ar@<-2pt>[r]&F^*\,\eqvee_{U} (U\cap V)}$ is the kernel star of the regular epimorphism $f'$. Finally, it is clear that $F^*\,\eqvee_{U} M$ is exactly the supremum of $L$ and $(G^*\cap (U\times U))^*$ (as stars on $F^*\,\eqvee_{U} (U\cap V)$) obtained by taking the kernel star of the diagonal in the right-hand pushout of regular epimorphisms $f'$ and $g'_1$ above) required in the non-pointed version of the Zassenhaus Lemma given in \cite{wT}.

\begin{corollary}
In the pointed context, let $\Cc$ be a normal category with pushouts of regular epimorphisms. Let $K\longrightarrow U, L\longrightarrow V$ be two kernels and $U\longrightarrow A, V\longrightarrow A$ be two monomorphisms, then we have the isomorphisms
\begin{eqnarray*}
\dfrac{K\,\eqvee_{U} (U\cap V)}{K\,\eqvee_{U}((K\cap V)\vee_{U\cap V} (L\cap U))}
&\cong & \dfrac{U\cap V}{(K\cap V)\vee_{U\cap V} (L\cap U)} \\
                       & \cong &\dfrac{L\,\eqvee_{V} (U\cap V)}{L\,\eqvee_{V}((K\cap V)\vee_{U\cap V} (L\cap U))} 
\end{eqnarray*}
\end{corollary}

The Zassenhaus Lemma (Theorem \ref{ZL}) is summarized in figure \ref{diag}
\begin{figure}
\[
\xymatrix @!0 @C=36mm @R=15mm{&&A\\B&&&&C\\U\ar[u]^f\ar[rruu]^u&&T&&V\ar[u]_g\ar[lluu]_v\\&f(U\cap V)\ar[luu]\ar[ru]^{\overline{g_1}}&&g(U\cap V)\ar[ruu]\ar[lu]_{\overline{f_1}}\\&F^*\,\eqvee_{U} (U\cap V)\ar[luu]\ar[u]\ar[ruu]_(0.6){t_1}&&G^*\,\eqvee_{V} (U\cap V)\ar[luu]^(0.6){t_2}\ar[u]\ar[ruu]\\&&U\cap V\ar[lu]^{y'}\ar[luu]_{f_1}\ar[uuu]^q\ar[ruu]^{g_1}\ar[ru]_{x'}\\&F^*\,\eqvee_{U} M\ar@<2pt>[uu]\ar@<-2pt>[uu]&&G^*\,\eqvee_{V} M
\ar@<2pt>[uu]\ar@<-2pt>[uu]\\F^*\ar@<2pt>[uuuuu]\ar@<-2pt>[uuuuu]&&M\ar@<2pt>[uu]
\ar@<-2pt>[uu]\ar[lu]^y\ar[ru]_x&&G^*\ar@<2pt>[uuuuu]
\ar@<-2pt>[uuuuu]\\&(F^*\cap (V\times V))^*\ar[lu]^{v_1}\ar[ru]_(0.6){h_1}&&(G^*\cap (U\times U))^*\ar[lu]^(0.6){h_2}\ar[ru]_{u_1}}
\]
\caption{Diagram of the Zassenhaus Lemma}\label{diag}
\end{figure}
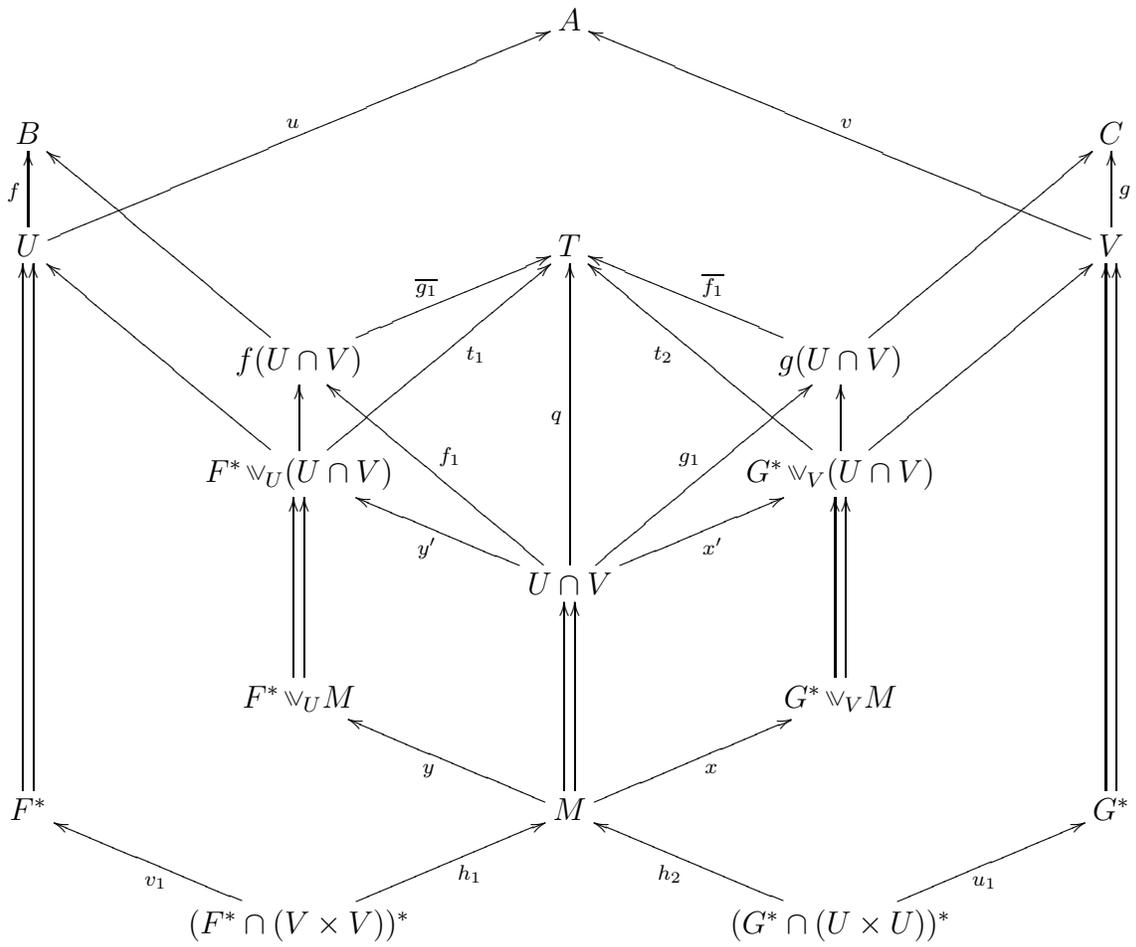

In the last part of this section we are interested in the \textit{categories with a good theory of ideals} (see \cite{GranJaneUrs} for more details). A \textit{category with good theory of ideals} is a star-regular category where kernel stars are stable under regular images. 

This definition unifies the notion of ideal determined category and Barr-exact Goursat category. An \textit{ideal determined category} \cite{JaneMarThoUrs} is a normal category (with binary coproducts) in which the normal image of a normal monomorphism is again a normal monomorphism. 
A category $\Cc$ is called a\textit{ Goursat category} when $\Cc$ is regular and the equivalence relations in $\Cc$ are 3-permutable, i.e. $RSR = SRS$ for any pair of equivalence relations R and S on the same object \cite{CKP}. This property is equivalent to asking that the regular image of any equivalence relation is again an equivalence relation. 

In the pointed context, a star-regular category
with finite coproducts $\Cc$ has a good theory of ideals if
and only if $\Cc$ is an ideal determined category. 

In the total context, a star-regular category $\Cc$ has a good theory of ideals if and only if $\Cc$ is a Barr-exact Goursat category.

When $\Cc$ is a category with a good theory of ideals ({so that} $\Cc$ has pushouts of regular epimorphisms and obviously satisfies the property $(*)$), there is some simplifications, one has that for two kernel stars $F^*\rightrightarrows A$ and $G^*\rightrightarrows A$ in $\Cc$, we have $$F^*\,\eqvee_{\!A}\, G^*=F^*\vee_A G^*=G^*\vee_A F^*=G^*\,\eqvee_{\!A}\, F^*.$$

Moreover, for two kernel stars $\kappa\colon \xymatrix@=16pt{F^*\ar@<2pt>[r]\ar@<-2pt>[r]&U}$, $\sigma\colon \xymatrix@=16pt{G^*\ar@<2pt>[r]\ar@<-2pt>[r]&V}$ and two monomorphisms $u\colon U\rightarrow A$, $v\colon V\rightarrow A,$ we have

$$ F^*\,\eqvee_{U} ((F^*\cap (V\times V))^*\vee_{U\cap V} (G^*\cap (U\times U))^*) \cong  F^*\,\eqvee_{U} (G^*\cap (U\times U))^*                      $$
and
$$G^*\,\eqvee_{V} ((F^*\cap (V\times V))^*\vee_{U\cap V} (G^*\cap (U\times U))^*)\cong  G^*\,\eqvee_{V} (F^*\cap (V\times V))^*. $$


The Zassenhaus Lemma then takes the following form, which is similar to the classical one in group theory 
\begin{eqnarray*}
\dfrac{F^*\,\eqvee_{U} (U\cap V)}{F^*\,\eqvee_{U}(G^*\cap (U\times U))^*}
&\cong & \dfrac{U\cap V}{(F^*\cap (V\times V))^*\vee_{U\cap V} (G^*\cap (U\times U))^*} \\
                       & \cong &\dfrac{G^*\,\eqvee_{V} (U\cap V)}{G^*\,\eqvee_{V} (F^*\cap (V\times V))^*} 
\end{eqnarray*}

\begin{corollary}\label{ZLPC}
In the pointed context, let $\Cc$ be an ideal determined category. Let $K\longrightarrow U, L\longrightarrow V$ be two kernels and $U\longrightarrow A, V\longrightarrow A$ be two monomorphisms, then we have the isomorphisms
\begin{eqnarray*}
\dfrac{K\,\eqvee_{U} (U\cap V)}{ K\,\eqvee_{U}(L\cap U)}
\cong  \dfrac{U\cap V}{(K\cap V)\vee_{U\cap V} (L\cap U)}
                        \cong \dfrac{L\,\eqvee_{V} (U\cap V)}{L\,\eqvee_{V}(K\cap V)}. 
\end{eqnarray*}
\end{corollary}

In $\mathsf{Grp}$, let $H_1$, $H_2$, $N_1$, $N_2$ be four subgroups of $G$ such that $N_1$ is normal in $H_1$ and $N_2$ is normal in $H_2$. Then Corollary \ref{ZLPC} becomes the well-know Zassenhaus Lemma in group theory
 \begin{eqnarray*}
\dfrac{N_1(H_1\cap H_2)}{ N_1(N_2\cap H_1)}
\cong  \dfrac{H_1\cap H_2}{(N_1\cap H_2)(H_1\cap N_2)}
                        \cong \dfrac{N_2 (H_1\cap H_2)}{N_2(N_1\cap H_2)}. 
\end{eqnarray*}

\section{An example: cocommutative Hopf algebras }
In this section, we apply Theorem \ref{ZL} to the category of cocommutative Hopf algebras. \\

Recall that a coassociative and counital coalgebra over a field $K$, called $K$-coalgebra,  is a vector space $C$ with linear maps $\Delta\colon C\to C\ot C$ and $\epsilon\colon C\to K$ satisfying $(id\ot \Delta)\cdot \Delta=(\Delta\ot id)\cdot\Delta$ (coassociativity) and $(id\ot \epsilon)\cdot \Delta=id=(\epsilon\ot id)\cdot\Delta$ (counitality). We use the classical Sweedler notation for the comultiplication,
and we write $\Delta(c)=c_1\ot c_2$ for any $c\in C$ (with the usual summation convention, where $c_1\ot c_2$ stands for $\sum c_1\ot c_2$, \cite{Sweedler}). 

With those notations, the coassociativity and the counitality are expressed, respectively, by the formulas
\begin{eqnarray*}
&c_1\ot c_{2,1}\ot c_{2,2}=c_{1,1}\ot c_{1,2}\ot c_2=c_1\ot c_2\ot c_3,\\
&c_1\epsilon(c_2)=c=\epsilon(c_1)c_2.
\end{eqnarray*}
A $K$-bialgebra $(A,M,u,\Delta,\epsilon)$ is given by an algebra $A$ with multiplication $M\colon A \ot A\to A$ and unit $u\colon K\to A$ which is also a coalgebra with comultiplication $\Delta\colon A\to A\ot A$ and counit $\epsilon\colon  A\to K$ such that $M$ and $u$ are coalgebra morphisms. In Sweedler's notation, these conditions are expressed for any $a,b \in A$ as
\begin{eqnarray*}
(ab)_1\ot (ab)_2 = a_1b_1\ot a_2b_2 && 1_1\ot 1_2=1\ot 1\\
\epsilon(ab)=\epsilon(a)\epsilon(b) && \epsilon(1)=1.
\end{eqnarray*}
A Hopf $K$-algebra is a sextuple $(A,M,u,\Delta,\epsilon,S)$, where $(A,M,u,\Delta,\epsilon)$ is a bialgebra endowed with a linear map $S\colon A\to A$, called the {\em antipode}, making the following diagrams commute
\[\xymatrix{
& A\otimes A \ar@<0.5ex>[rr]^-{S \otimes id} \ar@<-0.5ex>[rr]_-{id \otimes S} & &  A\otimes A \ar[dr]^-{M}  \\
A \ar[rr]_-{\epsilon} \ar[ur]^-{\Delta} & &K \ar[rr]_-{u}& & A.
   }\]
In the Sweedler notation, the commutativity of these diagrams can be written as
$$a_1S(a_2)=\epsilon(a)1_A=S(a_1)a_2,$$
for any $a\in A$.

A Hopf algebra $(A,M,u,\Delta,\epsilon,S)$ is \textit{cocommutative} if its underlying coalgebra is cocommutative, meaning that the comultiplication map $\Delta$ satisfies $\sigma \cdot \Delta = \Delta$, where $\sigma\colon A \otimes A \longrightarrow A \otimes A$ is the switch map $\sigma(a \otimes b)= b \otimes a$, for any $a\otimes b\in A \otimes A$. In the Sweeder notation: $a_1\ot a_2=a_2\ot a_1$.

A morphism of Hopf algebras is a linear map that is both an algebra and a coalgebra morphism (the antipode is automatically preserved). 
$\mathsf{Hopf_{K,coc}}$ is the category whose objects are cocommutative Hopf $K$-algebras and whose morphisms are morphisms of Hopf $K$-algebras. 

It was proven in \cite{GSV} that this category is semi-abelian (see also \cite{GKV} for a special case of this result). So, in particular $\mathsf{Hopf}_{K,coc}$ is a normal category which is also ideal determined. Thanks to the Lemma \ref{lemma sup}, we know that in a semi-abelian category the asymmetric join of a kernel $k \colon  K \rightarrow A$ and a monomorphism $m \colon  M \rightarrow A $ is given by the supremum of $K$ and $M$ (as subobjects of $A$). Let us prove some intermediate results before giving the description of such a supremum in the category of cocommutative Hopf algebras.

We recall \cite{A} that a Hopf subalgebra $K$ of $A$ is a\textit{ normal Hopf subalgebra} if $K$ is stable for the left and right adjoint actions ($a_1kS(a_2) \in K$ and $S(a_1)ka_2 \in K$ for any $k \in K$ and $a \in A$)

\begin{lemma}\label{normal}
Let $M$ and $K$ be two Hopf subalgebras of a Hopf algebra $A$, with $K$ normal in $A$. Then the vector spaces $KM$ and $MK$ are equal, where $KM$ and $KM$ are the vector spaces generated by the sums of elements in $ \{ km \;  | \;  k \in K , m \in M \} $ and $ \{ mk \;   | \;  k \in K , m \in M \} $ respectively. 
\end{lemma}
\begin{proof}
Let $k$ be an element of $K$ and $m$ be an element of $M$, then
$$ km = m_1 (S(m_2) k m_3)\in MK
$$and
$$ mk = (m_1 k S(m_2))m_3  \in KM $$ since $K$ is normal in $A$.
 \end{proof}

\begin{lemma}
Let $M$ and $K$ be two Hopf subalgebras of a (cocommutative) Hopf algebra $A$, with $K$ normal in $A$. Then the vector space $KM$ is a (cocommutative) Hopf algebra.
\end{lemma}
\begin{proof} 
Consider $k,h \in  K$ and $n,m \in M$, thanks to the previous lemma,\\
$KM$ is a subalgebra: 
$$knhm  \in KMKM = KKMM = KM. $$ 
$KM$ is a subcoalgebra: 
$$ \Delta(km) = k_1m_1 \otimes k_2m_2 \in KM \otimes KM $$ 
$KM$ is Hopf subalgebra: 
$$S(km) = S(m)S(k) \in MK =  KM $$  The cocommutativity is clear. 
\end{proof}
We recall that when it exists, the supremum $K \vee M$ of two subobjects $K$ and $M$ of the same object $A$ is the smallest subobject of $A$ containing $K$ and $M$. In other words, for any object $L$ such that $K$ and $M$ factorise through $L$, the supremum also factorises through it. In particular, if the category has binary coproducts and any arrow has a factorisation as a strong epimorphism followed by a monomorphism, the supremum always exists. 
\begin{proposition}
Let $M$ and $K$ be two Hopf subalgebras of $A$ with $K$ normal in $A$. Then $KM$ is the supremum of $K$ and $M$ as Hopf subalgebras.

\end{proposition}

\begin{proof}
It is clear that if there exists a Hopf subalgebra $L$ containing $K$ and $M$, then $L$ contains also $KM$ (since a Hopf subalgebra is closed under products and sums).
\end{proof}

In conclusion, in the specific context of Hopf algebras, for any kernel $k \colon  K \rightarrow A$ and any monomorphism $m \colon  M \rightarrow A$,  $$K \rotatebox[origin=c]{270}{$\geqslant$}_A M = KM.$$
 
Before applying the results obtained in star-regular categories to $\mathsf{Hopf}_{K,coc}$, we recall the construction of the quotient in this category. For a Hopf algebra $X$ we write $X^+$ for $\{x \in X | \epsilon(x) = 0\}$.  The cokernel of an arrow $f \colon A \rightarrow B$ in $\mathsf{Hopf}_{K,coc}$ is given by the canonical quotient 
$$q \colon B \rightarrow B/Bf(A)^+B,$$
where $f(A) = \{ f(a) \, \mid \, a \in A \}$ is the direct image of $A$ along $f$. 
Remark that if we compute the cokernel of a normal monomorphism $K \rightarrow B$, the cokernel is given by $q \colon  B \rightarrow B/BK^+$ thanks to the Lemma \ref{normal}. For more details about basic properties of Hopf algebras we refer to \cite{Sweedler, A}.  \\

Thanks to the above descriptions we can apply the results obtained in the framework of star-regular categories in the pointed context to cocommutative Hopf algebras. 
\begin{proposition}[Diamond isomorphism theorem]
Let $A$ be a cocommutative Hopf algebra, $M$ a Hopf subalgebra of $A$ and $K$ a normal Hopf subalgebra of $A$. We have the following isomorphism
$$\frac{M}{M(K \cap M)^+} \cong \frac{KM}{KM(M)^+} = \frac{KM}{KM^+} $$
\end{proposition}

\begin{proposition}[Double quotient isomorphism theorem] Let $K,N$ be two normal Hopf subalgebras of a cocommutative Hopf algebra $A$ such that $K \subseteq N$. We have the isomorphism
$$ A/L \cong\frac{A/K}{L/K}$$
\end{proposition}

Notice that the above propositions can be also deduced from the results in \cite{Natale} (in a more general context). 

\begin{theorem}[Zassenhaus Lemma]\label{ZL in coco}
In $\mathsf{Hopf}_{K,coc}$, let $K$ be a normal Hopf subalgebra of $U$, $L$ a normal Hopf subalgebra of $V$, and $U$, $V$ two Hopf subalgebras of $A$. There is the following isomorphism
$$ \frac{K(U \cap V)}{K(U \cap V)(K(L \cap U ))^+} \cong \frac{L(U \cap V)}{L(U \cap V)(L(K \cap V))^+}.$$
\end{theorem}

This theorem is the analogue, for cocommutative Hopf algebras, of Theorem 3.10 of S.~Natale \cite{Natale}, who proved a similar result for finite dimensional Hopf algebras. Thanks to the above result we now have a larger class of Hopf algebras in which the Zassenhaus Lemma is satisfied. It seems natural to try to unify and to extend these two versions of the Zassenhaus Lemma to a common categorical framework.



\begin{thebibliography}{10}
\bibitem{A} N.~Andruskiewitsch, \newblock \emph{Notes on extensions of Hopf algebras}. \newblock {\em Canad. J. Math.} 48 (1), 3--42, 1996.
\bibitem{BGO}M.~Barr, P.A.~Grillet, D. H. van Osdol, \emph{Exact categories and categories of sheaves}, Springer Lecture Notes in Mathematics 236, 1971.
\bibitem{BG} F.~Borceux, M.~Grandis, \emph{Jordan-H\"older, Modularity and Distributivity in Non-Commutative Algebra}, J. Pure Appl. Alg. 208-2, 665--689, 2007.
\bibitem{Bou02} D.~Bourn, \emph{The denormalized $3\times 3$ lemma}, J. Pure Appl. Alg. 177, 113-129, 2003.
\bibitem{CKP} A.~Carboni, G.M.~Kelly, M.C.~Pedicchio,
\emph{Some remarks on Maltsev and Goursat categories}, Appl. Categ. Struct. 1, 385-421, 1993.
\bibitem{Ehr} C.~Ehresmann, \emph{Sur une notion g\'en\'erale de cohomologie}, C. R. Acad. Sci. Paris 259, 2050-2053, 1964.
\bibitem{GE} T.~Everaert, M.~Gran, \emph{Monotone-light factorisation systems and torsion theories}, Bulletin des Sciences Math\'ematiques 137-8, 996-1006, 2013.
\bibitem{GranJaneRod} M.~Gran, Z.~Janelidze, D.~Rodelo,
\emph{$3\times 3$ Lemma for star-exact sequences}, Homology, Homotopy Appl. 14-2, 1-22, 2012.
\bibitem{GranJaneUrs} M.~Gran, Z.~Janelidze, A.~Ursini,
\emph{A good theory of ideals in regular multi-pointed categories}, J. Pure Appl. Alg. 216, 1905-1919, 2012.
\bibitem{GJRU} M.~Gran, Z.~Janelidze, D.~Rodelo, A.~Ursini, \emph{Symmetry of regular diamonds, the Goursat property, and subtractivity}, Th. Appl. Categ.
27, 80-96, 2012.
\bibitem{GKV}
M.~Gran, G.~Kadjo, and J.~Vercruysse.
\newblock \emph{A torsion theory in the category of cocommutative Hopf algebras}. \newblock {\em Appl. Categ. Struct.}, 24, 269-282, 2016.

\bibitem{GN} M.~Gran, O.~Ngaha Ngaha, \emph{Effective descent morphisms in star-regular categories}, Homology, Homotopy Appl. 15-2, 127-144, 2013.
\bibitem{GSV} M.~Gran, F.~Sterck and J.~Vercruysse \newblock \emph{A semi-abelian extension of a theorem by Takeuchi.} \newblock J. Pure Appl. Algebra,  Vol. 223, no. 10, 4171-4190, 2019.
\bibitem{Grandis} M.~Grandis, \emph{On the categorical foundations of homological and homotopical
algebra}, Cah. Top. G\'eom. Diff. Cat\'eg. 33, 135-175, 1992.
\bibitem{JMT} G.~Janelidze, L.~M\'arki, W.~Tholen, \emph{Semi-abelian categories}, J. Pure. Appl. Alg. 168, 367-386, 2002.
\bibitem{JaneMarThoUrs} G.~Janelidze, L.~M\'arki, W.~Tholen,
A. Ursini, \emph{Ideal determined categories}, Cah. Top. G\'eom.
Diff. Cat\'eg. 51, 113-124, 2010.
\bibitem{ZJane}Z.~Janelidze, \emph{The pointed subobject functor,
$3\times 3$ lemmas and subtractivity of spans}, Th. Appl. Categ.
23, 221-242, 2010.
\bibitem{MB} S.~Mac Lane, G. Birkhoff, \emph{Algebra} Second edition. Macmillan, Inc., New York; Collier-Macmillan Publishers, London, 1979.
\bibitem{Natale} S.~Natale, \newblock \emph{Jordan-H\"older Theorem for finite dimensional Hopf algebras}. \newblock Proc. Amer. Math. Soc. 143, 5195-5211, 2015.
\bibitem{Sweedler} M.E.~Sweedler, \newblock \emph{Hopf algebras}, Benjamin New York, 1969.
 \bibitem{Ttalk} W.~Tholen, \emph{Categorical algebra without axioms?}, Talk in CATALG2011, Gargnano (\url{http://math.unipa.it/metere/Gargnano2011/index.html}) 2011.
\bibitem{wT} W.~Tholen, \emph{Relative Bildzerlegungen und algebraische Kategorien}, Phd Thesis, 1974.
\bibitem{oW} O.~Wyler, \emph{The Zassenhaus lemma for categories}, Arch.  Math. XXII, 561-569, 1971.


\end{thebibliography}
\end{document}